\documentclass[10pt,a4paper, final]{amsart}
\usepackage[english]{babel}
\usepackage{pst-grad} 
\usepackage{pst-plot} 
\usepackage{pstricks}
\usepackage{amsmath,amssymb,mathrsfs,amsthm}
\usepackage[dvips]{graphicx}
\usepackage{epsfig}
\usepackage{showkeys}
\usepackage[
top=2cm, left=2cm, right=2cm, bottom=2cm]{geometry}

\newcommand{\eps}{\varepsilon}
\newcommand{\vhi}{\varphi}

\newcommand{\er}{\mathbb R}
\newcommand{\Ss}{\mathbb S}
\newcommand{\So}{{{\mathbb S}^1}}

\newcommand{\ZZ}{\mathbb Z}

\newcommand{\T}{\mathbb T}

\newcommand{\pat}{\partial_t}

\newcommand{\pax}{\partial_x}
\newcommand{\paxi}{\partial_{\xi}}

\newcommand{\vertiii}[1]{{\left\vert\kern-0.25ex\left\vert\kern-0.25ex\left\vert #1 
    \right\vert\kern-0.25ex\right\vert\kern-0.25ex\right\vert}}

\newcounter{comentcount}
\newcounter{teocount}
\newtheorem{lem}{Lemma}
\newtheorem{prop}{Proposition}

\newtheorem{theo}[teocount]{Theorem}  

\newtheorem{mydef}{Definition}

\newtheorem{rem}{Remark}

\title[Critical Keller-Segel meets Burgers on $\Ss^1$]{Critical Keller-Segel meets Burgers on $\Ss^1$. Large-time smooth solutions.}

\author[J. Burczak]{Jan Burczak}
\email{jb@impan.pl}
\address{Institute of Mathematics, Polish Academy of Sciences, \'Sniadeckich 8, Warsaw, 00-656, Poland}

\author[R. Granero-Belinch\'{o}n]{Rafael Granero-Belinch\'{o}n}
\email{granero@math.univ-lyon1.fr}
\address{Universit\'e Claude Bernard Lyon 1, CNRS UMR 5208, Institut Camille Jordan, 43 blvd. du 11 novembre 1918, F-69622 Villeurbanne cedex, France}

\begin{document}

\def\Xint#1{\mathchoice
 {\XXint\displaystyle\textstyle{#1}}%
 {\XXint\textstyle\scriptstyle{#1}}%
 {\XXint\scriptstyle\scriptscriptstyle{#1}}%
 {\XXint\scriptscriptstyle\scriptscriptstyle{#1}}%
 \!\int}
 \def\XXint#1#2#3{{\setbox0=\hbox{$#1{#2#3}{\int}$}
 \vcenter{\hbox{$#2#3$}}\kern-.5\wd0}}
 \def\ddashint{\Xint=}
 \def\dashint{\Xint-}
 
 \subjclass[2010]{35B65; 	35R11, 35Q92, 35S11, 92C17}
 \keywords{parabolic-ellipic Keller-Segel, critical fractional diffusion, large-time regularity, asymptotics}

\begin{abstract}
We show that solutions to the parabolic-elliptic Keller-Segel system on $\Ss^1$ with critical fractional diffusion $(-\Delta)^\frac{1}{2}$ remain smooth for any initial data and any positive time. This disproves, at least in the periodic setting, the large-data-blowup conjecture by Bournaveas and Calvez \cite{BouCal}. As a tool, we show smoothness of solutions to a modified critical Burgers equation via a generalization of the ingenious method of moduli of continuity by Kiselev, Nazarov and Shterenberg \cite{KNS} over a setting where the considered equation has no scaling. This auxiliary result may be interesting by itself. Finally, we study the asymptotic behavior of global solutions corresponding to small initial data, improving the existing results.
\end{abstract}

\maketitle 


\section{Introduction}
For unknown functions $u,v$, we study the parabolic-elliptic Keller-Segel problem\footnote{Known also as the Smoluchowski-Poisson equations. We will use both names interchangeably: \emph{the parabolic-elliptic Keller-Segel} and \emph{the Smoluchowski-Poisson} equations/problem.} with the critical fractional diffusion $\Lambda u= (-\Delta)^\frac{1}{2} u$ on the circle group $\Ss^1$. Alternatively, one can think of a $1$d periodic torus  $\T=[-L, L]$. Unknown $u$ follows
\begin{equation}\label{1_ks_u}
\begin{aligned}
\pat u &=-  \Lambda u - {\chi} \pax (u \pax v) \quad \text{ in } \; (0, T) \times \Ss^1, \\
u (0) &= u_0,
\end{aligned}
\end{equation}
where $\chi >0$ is a parameter. 
Let us observe immediately that the formal integration of \eqref{1_ks_u} in space implies 
\begin{equation}\label{eqmass}
\dashint_{\Ss^1} u (x, t) dx = \dashint_{\Ss^1} u_0 (x) dx =: m,
\end{equation} \emph{i.e.} conservation of the \emph{mean mass} $m$. Unknown $v$ is governed by 
\begin{equation}\label{1_ks_v}
0= \pax^2 v  + u - m \quad  \text{ in } \;(0, T) \times  \Ss^1.
\end{equation}
Solutions to  \eqref{1_ks_v} are unique up to an additive constant. We choose it by requiring 
\begin{equation}\label{zeroF}
\dashint_{\Ss^1} v (x,t)dx = 0.
\end{equation}

 In this paper, we deal with the global-in-time existence of classical solutions to  \eqref{1_ks_u} - \eqref{1_ks_v}  and their asymptotic behavior. In particular, we will prove that every $L^2$-initial datum gives rise to a unique solution $(u,v)$ that remains smooth for any $T<\infty$. Moreover, if $\chi m$ is small enough, the solution $(u,v)$ tends to the homogeneous state $(m,0)$.


In order to deal with the introduced Keller-Segel system, we study the following modified fractional Burgers equation
\begin{equation}\label{eq:red}
\begin{aligned}
\pat Z &=- \Lambda Z + f Z  \pax Z    \quad \;\; \text{ in } \; (0, T) \times \Ss^1, \\
Z (0) &= Z_0,
\end{aligned}
\end{equation}
where $ f = f(t)$, $f: \er_+ \to \er$ is a given smooth function. 

In fact, we will use a simple yet powerful observation that solutions to our Keller-Segel system  \eqref{1_ks_u} -  \eqref{1_ks_v}  are given as derivatives of solutions to the modified fractional Burgers equation  \eqref{eq:red}.  An analogous relation has been already utilised in analysis of certain problems, compare Bodnar \& Vel\'azquez \cite{BodVel} or 
Biler, Karch, Lauren\c{c}ot \& Nadzieja  \cite{BKLNplane, BKLNdisc}.

\subsection{Outline of the paper}
Section \ref{sec:mres} gathers our main results. There, Theorem \ref{th} provides smooth and unique large-data solutions to the Keller-Segel equations \eqref{1_ks_u} -  \eqref{1_ks_v}, globally in time. Its analogue for the modified Burgers \eqref{eq:red} is Theorem \ref{th2}. Theorem \ref{asymptotics} provides steady-state asymptotics for \eqref{1_ks_u} -  \eqref{1_ks_v} with initial datum $u_0$ having its mean mass of a \emph{medium size\footnote{By \emph{medium size} we mean that a smallness condition is needed, but, to the best of our knowledge, it is less restrictive than conditions currently available in the literature.}}. Next, Section  \ref{sec:mot} presents motivations to study our problem as well as a review of certain known results related to fractional Keller-Segel, classical Keller-Segel and Burgers equations. There, one of our main goals was justifying that the obtained lack of blowup for \eqref{1_ks_u} -  \eqref{1_ks_v} is a subtle and by no means generic matter. Section \ref{sec:pre} contains a collection of needed definitions, including the notion of the half-laplacian and definitions of a weak solution, as well as some auxiliary results:  short-time solvability and a continuation criterion, with sketches of proofs or relevant references. The following three sections are devoted to proofs of our main theorems. In the concluding Section \ref{sec:conc} we provide comparison of the considered by us periodic case with the real-line case as well as some remarks on relation between our fractional equation and the classical $2$d radially symmetric Smoluchowski-Poisson system. We end mentioning three open questions.

\section{Main results}\label{sec:mres}
The notions used in the following formulations are defined in a standard way, compare Section  \ref{sec:pre}. Our main result that disproves, at least in the periodic setting, the large-data-blowup conjecture by Bournaveas \& Calvez \cite{BouCal} reads
\begin{theo}[Regularity of critical Keller-Segel]\label{th} 
Fix any $T< \infty$ and real $s \ge 0$. For any initial datum $u_0 \in H^s (\So)$ there exists a weak solution $(u,v)$ to  \eqref{1_ks_u} - \eqref{1_ks_v} that satisfies
\[
\begin{aligned}
u &\in C( [0,T]; H^s (\So)) \cap C^\infty ( (0, T) \times \So), \\
v &\in C( [0,T]; H^{s+2} (\So)) \cap C^\infty ( (0, T) \times \So).
\end{aligned}
\]
Furthermore, this solution is unique among weak solutions.
\end{theo}

As already mentioned, the following result is needed as a tool to prove Theorem \ref{th}.  Since it provides details of a generalization of the method of moduli of continuity by Kiselev, Nazarov and Shterenberg \cite{KNS} over equations with no scaling, it may be interesting by itself.
\begin{theo}[Regularity of critical modified Burgers]\label{th2}
Fix any $T < \infty$,  a real $s \ge 1$ and a smooth $f=f(t)$. For any initial datum $Z_0 \in H^{s} (\So)$, problem \eqref{eq:red} admits a smooth solution
\[
Z \in C ([0, T]; H^{s} (\So)) \cap C^\infty ( (0, T) \times \So).
\]
It is unique among weak solutions.
\end{theo}

Finally, we study the asymptotic behavior of solutions to \eqref{1_ks_u} - \eqref{1_ks_v}. Recall that $m$ denotes the initial mean mass, see  \eqref{eqmass}.
\begin{theo}[Steady state asymptotics of critical Keller-Segel]\label{asymptotics}
Let $u_0\in L^2 (\Ss^1)$ be a $2\pi-$periodic initial datum for the system \eqref{1_ks_u} -  \eqref{1_ks_v}. 
Assume further that \[\chi m<1.\] Then, there exist $ \Sigma>0$ such that the solution couple $(u,v)$ to \eqref{1_ks_u} -  \eqref{1_ks_v} verifies for $t>0$
$$
|(u-m)(t)|_{\alpha}+|v(t)|_{2+ \alpha} \leq \Sigma e^{ (1-\alpha) \left(-1+\chi m\right) t}
$$
with any $\alpha \in [0,1)$.
\end{theo}
\begin{rem}
Theorem \ref{asymptotics} can also be proved for a general period $2L$. Then, the smallness condition takes the form  $\chi m<C(L)$. 

In fact, one can obtain a counterpart of Theorem \ref{asymptotics} for any $\alpha \in \er_+$, iterating the procedure presented in its proof further on. We were interested in obtaining decays that imply decay of $|(u-m)(t)|_{\infty}$, since the control of the supremum norm is the focal point in Keller-Segel-type equations.
\end{rem}
Let us recall that  for the case $\chi=1$ a similar condition, namely 
$m< (2\pi)^{-2}$, 
was obtained in  Ascasibar, Granero-Belinch\'{o}n \& Moreno  \cite{AGM} with a different method.  Theorem \ref{asymptotics} relaxes that condition.


\section{Motivation}\label{sec:mot}
First, we justify studies of the fractional Keller-Segel equations from the perspective of applications. Next, we indicate why the considered problem is challenging and rather subtle analytically.
\subsection{Applicational interest}
From the perspective of mathematical biology, the equations  \eqref{1_ks_u} - \eqref{1_ks_v} are a $1$d model of behavior of microorganisms (with density $u$) attracted by a chemical substance (with normalized density $v$). The parameter $\chi>0$ quantifies the sensitivity of organisms to the chemical signal. In the original system by Keller \& Segel \cite{KelSeg70} and its classical variations (compare for instance the survey \cite{HilPai_hbk} by Hillen \& Painter) the natural motility of microorganisms is modeled by $- \Delta u$, instead of our $(-\Delta)^\frac{1}{2} u$.  However, skippy movements of shrimps provide a good heuristic counterexample to the choice of laplacian. Actually, there is a strong evidence, both theoretical and empirical, that feeding strategies based on a L\'evy process (generated in its simplest isotropic-$\alpha$-stable version  by $(-\Delta)^\frac{\alpha}{2} u$) are both closer to optimal ones and indeed used by certain organisms, especially in low-prey-density conditions. The interested reader can consult Lewandowsky, White \& Schuster \cite{Lew_nencki} for amoebas, Klafter,  Lewandowsky \& White \cite{Klaf90} as well as Bartumeus et al. \cite{Bart03} for microzooplancton, Shlesinger \& Klafter \cite{Shl86} for flying ants and Cole \cite{Cole} in the context of fruit flies. Surprisingly, even for groups of large vertebrates, their feeding behavior is argued to follow L\'evy motions, the fact referred sometimes as to the \emph{L\'evy flight foraging hypothesis}. See  Atkinson,  Rhodes, MacDonald \& Anderson \cite{Atk} for jackals, Viswanathan et al. \cite{Vnature} for albatrosses, Focardi, Marcellini \& Montanaro \cite{deers} for deers and Pontzer et al. \cite{hadza} for the Hadza tribe. 

Thus, using  $(-\Delta)^\frac{\alpha}{2}$ in the equation for the density $u$ is fully justified for modeling certain biological phenomena. Our focus on  $(-\Delta)^\frac{1}{2}$ and on the $1$d case will become clear in the next subsection. 


Let us finally remark that writing $v := - \phi$ in  \eqref{1_ks_u} -  \eqref{1_ks_v}, we obtain a system for $(u, \phi)$ that is important in mathematical chemistry, cosmology and gravitation theory. It is very similar in spirit to the Zeldovich approximation used in cosmology to study the formation of large-scale structure in the primordial universe, see the works by Ascasibar, Granero-Belinch\'on \& Moreno \cite{AGM} and Biler \cite{Bi1}. It is also connected with the Chandrasekhar equation for the gravitational equilibrium of polytropic stars, statistical mechanics and the Debye system for electrolytes, compare for example Biler \& Nadzieja \cite{BilNad94}.

\subsection{Mathematical interest}
\subsubsection{Results on the $1$d Smoluchowski-Poisson equation}
We begin with explaining our focus on the half-laplacian in the equation for $u$. On $\er$, the problem
\[
\begin{aligned}
\pat u &=-  \Lambda^\alpha u - {\chi} \pax (u \pax v), \\
0 &= \pax^2 v + u,
\end{aligned}
\]
where $\Lambda^\alpha = (-\Delta)^\frac{\alpha}{2}$, 
enjoys global-in-time, regular solutions for $\alpha > 1$ and blowups for  $\alpha < 1$.  More precisely, Escudero in \cite{Escudero} shows global regularity in the case $\alpha > 1$ for any initial datum $u_0 \in H^1$. Bournaveas \& Calvez in \cite{BouCal} generalize this result by allowing any $u_0$ from  $L^{1+\delta}, \delta>0$. More importantly, they provide a class of initial data that gives rise to the finite-time blowup for $\alpha <1$ as well as a smallness condition $| u_0 |_{L^\frac{1}{\alpha}} \le K (\alpha)$ implying the global existence for $\alpha \le 1$. Consequently, the case $\alpha=1$ seems critical. To the best of our knowledge, presently the sharpest result is contained in Ascasibar, Granero-Belinch\'{o}n \& Moreno \cite{AGM}. It says, for the periodic torus $[-\pi, \pi]$, that the condition $m \leq (2\pi)^{-2}$ implies global existence and convergence towards the homogeneous steady state.
This bound on the initial data for the case \eqref{1_ks_u} -  \eqref{1_ks_v} can be recovered from both  \cite{BG} by the authors and \cite{GO} by Granero-Belinch\'{o}n \& Orive-Illera, where  more general systems were considered.

Hence, the remaining unresolved case is $\alpha =1$ for initial data with large masses. For this case, Bournaveas \& Calvez \cite{BouCal} conjecture a blowup of solutions, for certain large data, upon a numerical evidence. In this context Theorem \ref{th} seems especially interesting, since it provides an analytical proof for the opposite. 

Let us finally remark that in \cite{BG2} we have obtained a series of regularity results for the doubly parabolic generalization of  \eqref{1_ks_u} -  \eqref{1_ks_v}. Moreover, in \cite{BG3A, BG4} we proved that in the case of a logistically-damped Keller-Segel system, even certain fractional diffusions below the `critical' one ($\alpha=d$, $d=1,2$)  yield global-in-time smooth solutions.

\subsubsection{Results on the fractional Smoluchowski-Poisson-type equations in higher dimensions}

Let us recall some results for systems of type  \eqref{1_ks_u} -  \eqref{1_ks_v} in higher dimensions. They generally concern a system
\begin{equation}\label{eq:KD}
\begin{aligned}
\pat u &=- \mu \Lambda^\alpha u -  \nabla \cdot (u \nabla v), \\
v &= K * u
\end{aligned}
\end{equation}
in $\er^d$, $d \ge 2$, where $K$ stands for a nonincreasing (\emph{i.e.} attractive) interaction kernel. In \cite{BilKarLau09} Biler, Karch \& Lauren\c{c}ot, under rather general assumptions on $K$ (that allow for the Newtonian potential in particular), show blowup of solutions in the inviscid ($\mu = 0$) or $\alpha \in (0, 1)$ cases. For $d=2$ this result was generalized over $\alpha \in (0, 2)$  in  \cite{LiRodZha10} by Li, Rodrigo \& Zhang\footnote{Hence, within the generators of isotropic-$\alpha$-stable L\'evy processes, \emph{i.e.}  $\Lambda^\alpha u$ with $\alpha \in (0,2)$, one can have $\alpha_0$ that divides between regimes of large data global-in-time regularity  and blowups only for $d<2$. This is another strong justification for our interest in the one-dimensional case.}. A study for $K = e^{-|x|}$ was performed by Li \& Rodrigo in \cite{LiRod09},  \cite{LiRod09a}, \cite{LiRod10}.  In particular,  in \cite{LiRod10} they show regularity for  $\alpha > 1$ and  for $\alpha =1$ with small data. In \cite{LiRod09} blowups for $\alpha \in (0,1)$ are given. 
Notice that this choice of $K$ implies that $v$ solves
\[
c_d(1-\Delta)^{\frac{d+1}{2}} v = u,
\]
where $c_d(1-\Delta)^{\frac{d+1}{2}}$ is the pseudo-differential operator given on the Fourier side by
\[
\widehat{c_d(1-\Delta)^{\frac{d+1}{2}} v} (\xi) = c_d(1 + |\xi|^2)^\frac{d+1}{2} \hat v (\xi).
\]
Let us mention here also Biler \& Wu \cite{BilWu09}, that provides for $d=2$ and $K =- \frac{1}{2 \pi} \ln |x|$ a local-in-time well-posedness result in critical Besov spaces for $\alpha \in (1, 2)$.  

\subsubsection{Results on the classical Smoluchowski-Poisson equation}
The most classical case $d=2$, $\alpha =2$ and $K = -\frac{1}{2 \pi} \ln |x|$ of \eqref{eq:KD} exhibits the smoothness/blowup dichotomy, depending on the size of the initial mass, with the threshold mass $\frac{8 \pi}{\chi}$. The literature here is abundant, so let us only mention the seminal results by J\"ager \& Luckhaus \cite{JagLuc92} and  Nagai \cite{Nagai95} as well as the concise, more recent note by Dolbeault \& Perthame \cite{DolPer06}, where the threshold mass $\frac{8 \pi}{\chi}$ is easy traceable. Interestingly, even in this most classical case a \emph{single} quantity that both gives local existence and blowup criterion is still not fully agreed upon, since for the local-in-time existence one needs to assume not only the finiteness of the initial mass. Currently, the best candidate seems to be the scaling-invariant Morrey norm, compare Biler, Cie\'slak, Karch \& Zienkiewicz \cite{BCKZ} and its references. 

We refer to Section \ref{sec:conc} for some observations on comparing the radially symmetric $2$d classical parabolic-elliptic Keller-Segel system with our problem.

\subsection{Results on fractional Burgers and related equations in $1$d}
 Kiselev, Nazarov \& Sheterenberg in \cite{KNS} develop a method of moduli of continuity for proving the regularity of the critical case \eqref{eq:red} with $f \equiv 1$ and provide a rather complete picture: one has global, regular solutions for $\alpha \ge 1$ and blowups for some initial data for  $\alpha < 1$. Let us recall here also the earlier work \cite{BilFunWoy98} by Biler, Funaki \& Woyczy\'nski. The method of moduli of continuity was used to solve the major regularity problem of the $2$d critical surface quasi-geostrophic equation, see the celebrated paper by Kiselev, Nazarov \& Volberg \cite{KNV}. 
 
Interestingly, certain equations that resemble \eqref{eq:red}  at the first glance behave totally different. For instance, 
Chae, C\'ordoba, C\'ordoba \& Fontelos \cite{CCCF} (see also Castro \& C{\'o}rdoba \cite{CC} and C\'ordoba, C\'ordoba \& Fontelos \cite{CCF}) considered
\begin{equation}\label{eqcccf}
\pat Z =-  \mu \Lambda^\alpha Z - \delta\pax (Z  H( Z )) - (1-\delta)\pax Z  H( Z ),
\end{equation}
where $H$ stands for the Hilbert transform and $\delta \in [0,1]$. For the inviscid case $\mu=0$, they show that \eqref{eqcccf} develops finite-time singularities in the whole range $0\leq \delta\leq 1$. At the same time, \emph{weak} solutions exist globally-in-time for $\delta\geq \frac{1}{2}$, compare Bae \& Granero-Belinch\'on \cite{BG3}. What's more interesting for our considerations, in the case $\mu>0$ and $\alpha=\delta=1$ equation \eqref{eqcccf} develops singularities for large data and remains regular for small data. Li and Rodrigo in \cite{LiRod11} considered \eqref{eqcccf} with $\delta =1$ and the opposite sign of nonlinearity, \emph{i.e.}
\begin{equation}\label{eqLR}
\pat Z =-  \mu\Lambda^\alpha Z + \pax (Z  H( Z ))
\end{equation}
and obtained blow up in the whole range of $0\leq \alpha\leq 2$. 

Hence, comparing the regularity result of Kiselev, Nazarov \& Shterenberg \cite{KNS} (and our Theorem \ref{th2}) for \eqref{eq:red} with the just-mentioned large-data blowup results for equations  \eqref{eqcccf} and \eqref{eqLR}, one realizes how decisive is the exact form of the nonlinearity for the regularity studies. More precisely, $\pm  \pax (Z  H( Z ))$ at  \eqref{eqcccf} with $\delta=1$ and at \eqref{eqLR} acts in fact as \emph{both} a semilinear diffusion $\pm  Z  \Lambda Z$  \emph{and} a nonlinear transport term $\pm   H( Z )  \pax Z$. Consequently, the sign of this nonlinearity turns out to be decisive, because it may produce either a forward or a backward-type nonlinear diffusion equation. On the other hand, the nonlinearity $\pm  Z  \pax Z$ of \cite{KNS} or of our \eqref{eq:red} is merely a nonlinear transport term. Hence it does not spoil the  critical-case global regularity for any sign $\pm$ (compare Theorem \ref{th2} and observe that $f$ at \eqref{eq:red} may be negative). Actually, the negative sign has a regularizing effect. In this context, let us recall that Granero-Belinch\'on \& Hunter \cite{GH} show for 
\begin{equation}\label{NLKS}
\partial_t Z=(\Lambda^{\gamma}-\epsilon\Lambda^{1+\delta})Z -Z\partial_x Z
\end{equation}
with $1>\delta>0$, $\gamma \in [0, 1+\delta)$, $0<\epsilon<1$, the global existence of solutions. Furthermore, these solutions develop spatio-temporal chaos and remain close to the bounded attractor. Notice that the linearized version of \eqref{NLKS} in Fourier space is
$$
\frac{d}{dt}\hat{Z}=(|\xi|^{\gamma}-\epsilon|\xi|^{1+\delta})\hat{Z}.
$$
Since the linear part of \eqref{NLKS} contains the backward diffusion $\Lambda^\gamma$, the linear term pumps energy into the lower Fourier modes. Consequently, the boundedness of the solutions to \eqref{NLKS} is due only to the nonlinear term $-Z\partial_x Z$. 

\section{Preliminaries}\label{sec:pre}
We use standard definitions for Lebesgue and Hilbert spaces,  writing for brevity $| \cdot |_{H^k} = | \cdot |_k$, where $k < \infty$; $k=0$ stands for the $L^2$ norm. The supremum norm is denoted with $| \cdot |_\infty$. It is important not to mistake $ | \cdot |_m$ for a finite $m$ (a Hilbert norm) and for $m=\infty$ (the supremum norm). We will sometimes suppress indication of the domain involved, when there is no danger of confusion.
\subsection{Half-laplacian}\label{ssec:lambda}
Take $\alpha \in (0, 2)$. For an $f\!:  \Omega \to \er$ we have by definition
$$
\widehat{ \Lambda^\alpha f}(\xi) =  \, |\xi|^\alpha \hat{f}(\xi)
$$
in Fourier variables, $\Omega$ being either $\So$ or $\er$. This operator is the infinitesimal generator of the isotropic-$\alpha$-stable L\'evy process. Let us focus on the half-laplacian case $\alpha = 1$. It admits the following equivalent kernel representations for $\Omega = \er$
\begin{equation}\label{1}
\begin{aligned}
\Lambda f(x)&= \frac{1}{\pi} \text{P.V.}\int_{\er}\frac{f(x)-f(y)}{|x-y|^{2}}dy,\\
\Lambda f(x)&= \left[\frac{d}{dh} (P_h * f) \right]_{|  h = 0} \; ,
\end{aligned}
\end{equation}
where $P_h = \frac{1}{\pi} \frac{h}{h^2 + y^2}$ is the Poisson kernel. The latter representation follows from the harmonic extension theorem, see the celebrated \cite{CafSil}  by Caffarelli \& Silvestre. 

It is common to use on $\So$ the kernel representation
\begin{equation}\label{2}
\Lambda f(x) = \frac{1}{\pi} \sum_{\gamma \in \ZZ} \text{P.V.}\int_{\So}\frac{f(x)-f(y)}{|x-y - 2L \gamma|^{2}}dy,
\end{equation}
that explicitly involves a chosen period $2L$, whereas the expressions \eqref{1} are standard for the real line case. In the particular case $L=\pi$, using complex analysis tools to add the previous series, we can write
\begin{equation}\label{1sin}
\Lambda f (x) =\frac{1}{2\pi}\text{P.V.}\int_{-\pi}^\pi \frac{f(x)-f(x-y)}{\sin^2\left(y/2\right)}dy.
\end{equation}

Nevertheless, we derive our main results using real-line expressions \eqref{1} for periodic function. This is admissible, since a $2L$-periodic, sufficiently smooth $u$ produces in \eqref{1} a $2L$-periodic $\Lambda f$ and gives there integrability at infinity. Consequently, both principal values at \eqref{1} and at \eqref{2} concern in fact only the behavior around the origin. Moreover, the regularity of the involved functions will pose no problem, since a local-in-time smoothness is used in the proofs of our main results. For functions with little regularity, for instance in proofs of local-in-time existence, one can rely on the Fourier-side definition. The advantage of using \eqref{1} becomes clear in the proofs. Besides, it allows to justify our reasoning for an arbitrary period $2L$ at once. 

\subsection{Weak solutions}\label{ssec:weak}

\begin{mydef}[Weak solutions to Keller-Segel system]
Choose $u_0 \in L^2(\So)$. Fix an arbitrary $T \in (0, \infty)$. The couple \[u\in C ([0,T), L^2(\So)),\qquad v \in C ([0,T), H^{2}(\So))\] is a \emph{weak solution} of   \eqref{1_ks_u} - \eqref{1_ks_v}  if and only if
\[
\begin{aligned}
\int_0^T\int_\So - u \, \pat\vhi + u \,  \Lambda \vhi  - ( \chi u\pax v) \,  \pax \vhi  &= \int_\So u_0 \, \vhi(x,0)  dx, \\
\pax^2 v (x,t) &= u(x,t)  - m \quad \text{ a.e. in } [0, T) \times \So, \\
\end{aligned}
\]
where $\vhi$ is an arbitrary, compactly supported\footnote{This is related only to lack of space integrals at $T$.} $C^\infty((-1,T)\times \So)$ function. 
\end{mydef}
Similarly, let us introduce
\begin{mydef}[Weak solutions to modified Burgers equation]
Choose $u_0 \in L^2(\So)$. Fix an arbitrary $T \in (0, \infty)$.  \[Z\in C ([0,T), L^2(\So))\] is a \emph{weak solution} of   \eqref{eq:red}  if and only if
\[
\int_0^T\int_\So - Z \, \pat\vhi  + Z \Lambda \vhi  + \frac{f}{2} Z^2 \, \pax \vhi = \int_\So Z_0 \vhi(x,0)dx,
\]
where $\vhi$ is an arbitrary compactly supported function from $C^\infty((-1,T)\times \So)$.
\end{mydef}


\subsection{Auxiliary results} 
In order to proceed with our global regularity proofs, we need to make sure that we are equipped with sufficient smoothness. To this end, the easiest way is to provide the following short-time regularity results for the modified Burgers equation \eqref{eq:red}. Concerning their proofs, for brevity, we restrict ourselves to providing the appropriate reference and commenting on some minor changes needed. Let us emphasize that from now on, we fix both an arbitrary $T < \infty$ and a particular smooth $ f(t)$, $f: \er \to \er$. Consequently, we do not write dependencies on $f$ and $T$ explicitly. 
Let us begin with  the unique, local-in-time existence of weak solutions.
\begin{lem}\label{th:loc}
Take $Z_0 \in H^s (\So)$  with an $s \ge 1$. Then, there exists $T_* = T_* (|Z_0|_s)$ such that problem \eqref{eq:red} admits a weak solution enjoying
\[
Z \in L^2 (0,T_*; H^{s + \frac{1}{2}}) \cap C( [0,T_*]; H^s).
\]
Moreover, this weak solution is unique in the class
\[
L^\frac{3}{2} (H^1) \cap C(L^2).
\]
\end{lem}
For $f \equiv 1$, this is Theorem 2.5 (existence of a weak solution) and Theorem 2.8 with $\delta = 1$ (uniqueness) of \cite{KNS}. Allowing for a non-constant $f (t)$ does not change these proofs, since smooth $f: \er \to \er $ and $T < \infty$ are prefixed.

Next, we state a simple continuation criterion 
\begin{lem}\label{burg:blowup}
Take a weak, local in time solution $Z$ to \eqref{eq:red} with its existence time $T_*>0$. If there exists $C$ such that $| Z(t)|_1 \le C$ on $[0,T_*)$, then for an  $\eps>0$, $Z$ can be continued beyond $T_*$ up to  $T_*+ \eps$.
\end{lem}
The proof comes down to restarting our evolution an instant before $T_*$ and use of Lemma \ref{th:loc}.

Finally, we arrive at a conditional smoothness result.
\begin{lem}\label{bur:analytic}
If $| Z(t)|_1 \le C$ on $[0,T)$, then the solution $Z$ to \eqref{eq:red} given by Lemma \ref{th:loc} satisfies
\[Z \in C^\infty ( (0, T) \times  \So).
\]
\end{lem}
Using $| Z(t)|_1 \le C$ on $[0,T)$ and Lemma \ref{burg:blowup}, we know that the weak solution to \eqref{eq:red} can be continued up to $T$. Now, smoothness  follows from Corollary 2.6 of \cite{KNS}, up to minor differences concerning  $f (t)$, that we can deal with as before. 

The above results are not optimal in relation  to underlying spaces: $H^s$ with $s \ge 1$ can be relaxed to $s > \frac{1}{2}$, compare \cite{KNS}. Such relaxation is not needed for our purposes, because $s=1$ here leads to the interesting for us $L^2$ data in the Keller-Segel system \eqref{1_ks_u} - \eqref{1_ks_v}.
\section{Proof of Theorem \ref{th}} \label{sec:proofs1}

The main steps of proof of Theorem \ref{th} are as follows
\begin{enumerate}
\item Restating considerations for Keller-Segel equations as considerations for a fractional Burgers-type equation \eqref{eq:red}. Roughly speaking, equation \eqref{eq:red} serves as an equation for the primitive function of $u$ solving \eqref{1_ks_u}.
\item Using smoothness of solutions to \eqref{eq:red} provided by Theorem \ref{th2}.
\end{enumerate}
Theorem \ref{th2} will be proved in Section \ref{sec:proofs2}. Here in Section \ref{sec:proofs1} we assume that we have it at our disposal.

To fix ideas, let us choose a reference point $0$ on $\So$ and denote its half-length with $L$, \emph{i.e.} we work with the periodic torus $\T = [-L, L]$. 

\subsection{A primitive-type equation}
Take $u_0 \in H^s$, $s \ge 0$. Recall that by definition $m = \dashint_{-L}^L   u_0 ( y) dy$. Let us choose
\begin{equation}\label{bur:choi}
\begin{aligned}
f(t) &:= \chi  e^{\chi m  t}   ,\\
Z (0, x) &:=   \int_{-L}^x u_0 ( y) dy - m (x + L) + c,
\end{aligned}
\end{equation}
where
\[
c =  L m - \dashint_{-L}^L \int_{-L}^x u_0 ( y) dy dx 
\]
is selected so that
\[
\int_{-L}^L Z (0, x) dx = 0.
\]
Observe that  $Z (0, x)$ is periodic and belongs to $H^{s+1}$. Hence,
 Theorem \ref{th2} implies that problem  \eqref{eq:red} with the choice \eqref{bur:choi} admits a unique, smooth solution 
\[Z \in C( [0,T]; H^{s+1} (\So)) \cap C^\infty ( (0, T) \times \So).
\]
By integration of \eqref{eq:red} in space we see that $\int_\So Z(x, t) dx \equiv 0$.

\subsection{Change of variables} Formula
\begin{equation}\label{bur:wz}
 W (x,t) :=  e^{\chi m t} Z (x,t)
\end{equation}
provides a one-to-one correspondence between solutions of  \eqref{eq:red} with the choice \eqref{bur:choi}  and 
\[W \in C( [0,T]; H^{s+1} (\So)) \cap C^\infty ( (0, T) \times \So)
\]
that solves
\begin{equation}\label{1_ks'}
\begin{aligned}
\pat W &=- \Lambda W + {\chi} (\pax W) W + {\chi}m W \quad  \text{ in } \;(0, T) \times  \So,\\
  W (0, x) &=   \int_{-L}^x u_0 ( y) dy - mx - \dashint_{-L}^L \int_{-L}^\xi u_0 ( y) dy d \xi.\\
\end{aligned}
\end{equation} 
The definition \eqref{bur:wz} of $W$ implies  $\int_\So W(x, t) dx \equiv 0$.
\subsection{Recovering solutions to Keller-Segel system}
Let us introduce \begin{equation}\label{smallW} u = \pax W+m.\end{equation} 
Smoothness of $W$ gives \[u \in C( [0,T]; H^s (\So)) \cap C^\infty ( (0, T) \times \So). \] It  solves 
\begin{equation}\label{1_ks'_a}
\begin{aligned}
\pat u &=- \Lambda u + \chi \pax ( u W)   \quad  \text{ in } \;(0, T) \times  \So,\\
w (0, x) &=  u_0 (x),
\end{aligned}
\end{equation} 
since $W$ solves \eqref{1_ks'} and $\pax$ and $\Lambda$ commute. Observe that $\dashint_\So u (x, t) dx \equiv m$.  Function $u$ of \eqref{1_ks'_a} gives in fact solution to  \eqref{1_ks_u}. It becomes fully clear after we recover $v$ solving \eqref{1_ks_v}. To this end, let $V$ be a solution to
\begin{equation}\label{largeV}
-\pax^2 V (x, t) = W (x,t)   \quad  \text{ in } \;(0, T) \times  \So,
\end{equation}
where  $W$ is an admissible right-hand side, because  $\int_\So W(x, t) dx \equiv 0$.
Hence $ v := \pax V$ is the zero-mean solution of
\begin{equation}\label{1_ks'_b}
-\pax^2 v = u - m   \quad  \text{ in } \;(0, T) \times  \So, 
\end{equation}
compare \eqref{smallW}. By the definition  $ v= \pax V$ and \eqref{largeV}, we get
$
 \pax v = -W.
$ Plugging this in  \eqref{1_ks'_a} and looking at \eqref{1_ks'_b} yields that \[
\begin{aligned}
u &\in C( [0,T]; H^s(\So)) \cap C^\infty ( (0, T) \times \So), \\
v &\in C( [0,T]; H^{s+2}(\So)) \cap C^\infty ( (0, T) \times \So)
\end{aligned}
\] solve \eqref{1_ks_u} - \eqref{1_ks_v} with the initial condition $u_0$. 

\subsection{Uniqueness} Since Theorem \ref{th2} provides uniqueness of $Z$ and \eqref{smallW} holds,  uniqueness of $u$ solving  \eqref{1_ks_u} follows. Since $v$ is the zero-mean solution of  \eqref{1_ks'_b}, it is also unique.

\section{Proof of Theorem \ref{th2}}\label{sec:proofs2}
In order to show Theorem \ref{th2} we will  follow the ingenious methodology developed in  \cite{KNS} by Kiselev, Nazarov and Shterenberg for the fractional Burgers equation. Our case differs from theirs by our lack of scaling and, secondly, by involving slightly stronger destabilizing term  $f Z\pax Z$ as compared to  $Z\pax Z$.

It could have  been sufficient to say that we follow the lines of the respective part of \cite{KNS}, whereas we deal with lack of staling by constructing \emph{at once} an entire family of moduli of continuity, so that any (sufficiently smooth) initial datum enjoys one of them. Nevertheless, we provide rigorous proofs in what follows. One motivation is completeness. A more important one is to provide details of the method for problems with no scaling, that may be useful in other applications. In this context let us remark the following technical difficulty. Using a one-parameter family of moduli of continuity as in   \cite{KNS}, which is the most natural approach suggested by the scaling of the full-space case, turned out to be insufficient. We need to introduce another parameter $N$, that divides the middle and large arguments of moduli of continuity (see formula \eqref{cpl_ini2}).   
\subsection{Reduction of the smoothness problem  to keeping a modulus of continuity.}
As indicated before, the proof relies on a construction of special moduli of continuity (in space). Hence let us first introduce
\begin{mydef}
We denote by $\mathcal{O}$ the class of moduli of continuity $\omega: \er_+ \to \er_+$ (non-decreasing, concave functions with $\omega(0)=0$), whose derivatives satisfy
\[\omega' (\xi) \text{ is continuous at } 0, \qquad  \lim_{\xi \to 0}  \omega'' (\xi) = -\infty.
\]
\end{mydef}
Recall that a function $f$ has modulus of continuity $\omega$ if
$$
|f(x+h)-f(x)|\leq \omega(|h|).
$$
It holds (compare Lemma 3.3 of \cite{KNS})
\begin{prop}\label{prop:1}
If a smooth real function $f$ has a modulus of continuity $\omega \in \mathcal{O}$, then
\begin{equation}\label{omega'}
|\partial_x f|_\infty < \omega' (0).
\end{equation}
\end{prop}
\begin{proof}
Assumption $|f(x+h)-f(x)|\leq \omega(|h|)$ and the Taylor formula implies for an arbitrary $x$
\[
\left|\partial_x f(x) h +\partial^2_\eta f (\eta) \frac{h^2}{2} \right|\leq \omega'(|h|)h + \omega'' (|\eta_1|) \frac{h^2}{2}
\]
with $\eta, \eta_1$ in-between $0$ and $|h|$, because $\omega (0) = 0$. Consequently
\begin{equation}\label{omega'a}
\left|\partial_x f(x) h  \right|\leq \omega'(|h|) h + \omega'' (|\eta_1|) \frac{h^2}{2} + \left| \partial^2_\eta f (\eta) \frac{h^2}{2}  \right|.
\end{equation}
 Since $ \lim_{\xi \to 0}  \omega'' (\xi) = -\infty$ and $\sup_{\eta \in [0, |h|]} |\pax^2 f (\eta)| < \infty$  (due to its smoothness) one may find such a small $h_0>0$  that   \[ \sup_{|\eta_1| \in [0, h_0]} \omega'' (|\eta_1|) + \sup_{\eta \in [0, h_0]} |\pax^2 f (\eta)| <0.\] 
This used in \eqref{omega'a} yields $|\partial_x f|_\infty h_0 < \omega' (h_0) h_0$. Division by $h_0$ and concavity of $\omega$ imply \eqref{omega'}.
\end{proof}
Next, let us show that if a solution of \eqref{eq:red} keeps a modulus of continuity $\omega \in  \mathcal{O}$, then it remains smooth.
\begin{lem}\label{reduction}
Assume that  a zero-mean solution $Z(t)$ to \eqref{eq:red} satisfies
$$
Z(t_0) \in C^\infty (\So)
$$
and there exists $\omega \in  \mathcal{O} $ such that 
$$
|Z (t,x) - Z (t,y)| \le \omega (|x-y|) \quad \text{ on } \quad [t_0, t_1] \times  \So,
$$
where $0 \le t_0 \le  t_1 < \infty$.
Then, there exists $\eps > 0$ such that $Z(t) \in C^\infty (\So)$ on $[t_0, t_1+ \eps]$.
\end{lem}
Recall that $| \cdot |_\infty$ denotes the supremum norm and $| \cdot |_0$ denotes the $L^2$ norm.
\begin{proof}
Once smooth $Z(t_0)$ complies with the modulus $\omega$, formula \eqref{omega'} yields \[
|\pax Z (t_0) |_\infty < \omega' (0).
\]
Hence \[
|\pax Z (t_0) |_0 < \omega' (0)  |\So|^\frac{1}{2}.
\]
The zero mean assumption implies then 
\[
| Z (t_0) |_1 < C \omega' (0)  |\So|^\frac{1}{2} =: D.
\]
Lemma \ref{th:loc} for the initial datum $Z (t_0)$ with Lemma \ref{bur:analytic} give that $Z(t)$ is smooth on $[t_0, t_0 + T_* (D)]$. It keeps the zero-mean over its evolution. By assumption, $Z (t)$ does not violate $\omega$ as long as $t_0 + T_* (D) \le t_1$. In such case, we restart our evolution at $t^{(1)}_0 := t_0 + T_* (D)$ and repeat our considerations. This gives smoothness of $Z$ up to $t^{(2)}_0 := t_0 + 2 T_* (D)$. We can continue this procedure up to $t^{(n+1)}_0$, where $t^{(n)}_0 < t_1 $ and $t^{(n+1)}_0 > t_1$.
\end{proof}
Notice also that the assumption of smoothness of $Z(t_0)$ is in fact not a restriction, since Lemma \ref{bur:analytic} implies that the Burgers problem \eqref{eq:red} enjoys a local-in-time smooth solution for every initial data $Z(0)\in H^1$.

\subsection{Smoothness} 
The following result  is the essential part of this section
 \begin{lem}\label{mainlemma}
Choose any $T < \infty$.
If for a $t_0 \in [0, T]$
\[
Z(t_0) \in C^\infty (\So),
\]
then $Z$ remains smooth on $[t_0, T]$, \emph{i.e.} $Z \in C^\infty([t_0, T] \times \So)$.
\end{lem}
Since the proof of Lemma \ref{mainlemma} is comparatively long, we subdivided it into several steps for clarity. In order to make it even more accessible, let us begin with
\subsubsection{Outline of the proof of  Lemma \ref{mainlemma} }
In view of Lemma \ref{reduction}, it is enough to find such $\omega \in \mathcal{O}$ that
\begin{equation}\label{cpl_ev22}
|Z (x, t) - Z(y, t)| \le \omega (|x-y|)
\end{equation}
for $t \ge t_0$.

First, assuming that $
|Z (x, t_0) - Z(y, t_0)| < \omega (|x-y|) $, we will argue that the only way in which $\omega$ can be broken over further time evolution of $Z$, is the \emph{two-point blowup scenario}: at certain time $\tau> t_0$ there exist two distinct points such that  $|Z (x, \tau) - Z(y, \tau)| = \omega (|x-y|)$. 

Next, we will derive inequalities that allow to falsify this scenario, provided a certain refinement of properties of  $\omega$ holds. This in turn will indicate what exact form of $\omega$ is suitable. Consequently, we will construct a specific family of moduli $\{ \omega_{K, B, \xi_0} \}$, such that \eqref{cpl_ev22} is satisfied for a certain choice of parameters ${K, B, \xi_0}$. We could have chosen a shorter way of presenting this point, namely introduce  immediately  the family $\{ \omega_{K, B, \xi_0} \}$ and show that $\omega_{K, B, \xi_0}$ is kept over the evolution. Then, however, there would be no clear rationale for using the  concrete shape of family of moduli of continuity.

Let us notice that we can work here with smooth solutions, which is not \emph{a priori} estimate. Namely, for short times we have smoothness via Lemma \ref{bur:analytic}. Next, as long as a modulus of continuity is conserved, we have smoothness by Lemma \ref{reduction}.
\subsubsection{Admissible blowup scenarios} 
Let us consider any $\omega \in \mathcal{O}$ such that 
\begin{equation}\label{cpl_ev0}
|Z (x, t_0) - Z(y, t_0)| < \omega  (|x-y|).
\end{equation}
Assume that 
\begin{equation}\label{ass:au}
Z \; \text{ looses the modulus of continuity $\omega$ at a certain time } t \in (t_0, T].
\end{equation} Let us take  such $\tau$ that 
\begin{equation}\label{eq:tau}
\tau := \sup \{\tau \in [t_0, T]:  |Z (x, t) - Z(y, t)| \le \omega (|x-y|) \}.
\end{equation}
There are now two scenarios of what happens at time $\tau$. Either we have the \emph{two-point blowup scenario}:
\begin{equation}\label{eq:twopt}
\exists y \neq x \text{ such that } |Z (x, \tau) - Z(y, \tau)| = \omega (|x-y|)
\end{equation}
or  we have the \emph{single-point blowup scenario} otherwise. 
\subsubsection{Ruling out the single-point blowup scenario} 
Roughly speaking, the  single-point blowup scenario is realized when for $t_n \to \tau$, $x_n \to x$, $\xi_n \to 0$ holds
\[
|Z (x_n, t_n) - Z(x_n + \xi_n, t_n)| \le \omega (|\xi_n|),
\]
with the above inequality becoming equality in the limit $n \to \infty$. Equivalently,
\[
|\pax Z (x_n + \tilde \xi_n, t_n) \xi_n | \le \omega' (|\bar \xi_n|) \xi_n,
\]
where $\tilde \xi_n, \bar \xi_n \in [0, \xi_n]$ and the inequality becomes equality in the limit $n \to \infty$. Hence it is indeed a single-point blowup, because in the limit we get
\[
|\pax Z ( x, \tau) | = \omega' (0).
\]
This is impossible due to formula \eqref{omega'}. 

Now, let us provide a rigorous argument for this impossibility of the {single-point blowup scenario} that differs slightly from the original Lemma 3.4 of \cite{KNS}: Our choice \eqref{eq:tau} of $\tau$ and Lemma \ref{reduction} says that $Z(t)$ is smooth an $\eps$ beyond $\tau$. As a consequence, we can use formula \eqref{omega'} to get $|\pax Z (t)|_\infty < \omega' (0)$ for all $t \in [\tau, \tau + \eps]$. This sharp inequality and the fact that $\omega' (\xi)$ is continuous at $0$ implies that there exists such $\delta>0$ that $\sup_{t \in [\tau, \tau + \eps]} |\pax Z (t)|_\infty < \omega' (\delta)$. Hence for $t \in [\tau, \tau + \eps]$ and $|x-y| \le \delta$
\[
|Z (x, t) - Z(y, t )| \le |\pax Z (t)|_\infty |x- y| < \omega' (\delta) |x- y| \le  \omega( |x- y| ),
\]
where in the last inequality we have used the concavity and the mean value theorem. Consequently, the solution can loose the modulus of continuity only in the two-point blowup scenario \eqref{eq:twopt}.
\subsubsection{Towards moduli ruling out the two-point blowup scenario} Here, let us consider any $\omega \in \mathcal{O}$ smooth enough. We will add assumptions on $\omega$ useful for ruling out the two-point blowup scenario, thus approaching the more concrete shape of needed moduli of continuity. A comment on a technicality related to a (lack of) smoothness of the eventually constructed family of moduli of continuity will be presented in one of the following steps.

Without loss of generality we can assume $Z (x, \tau) \ge Z(y, \tau)$,  drop the absolute value in  \eqref{eq:twopt}, hence considering
\begin{equation}\label{eq:twopt'}
Z (x, \tau) - Z(y, \tau) = \omega (|x-y|).
\end{equation}
We have arrived at the heart of the regularity proof. To contradict  \eqref{ass:au}, it suffices to show that 
\begin{equation}\label{cd}
\partial_t (Z (x, t) - Z (y, t))_{|t=\tau} < 0,
\end{equation}
because then we arrive at contradiction with the definition \eqref{eq:tau} of $\tau$. 

Since we work at the exact time instant $\tau$, let us suppress the time dependence below.  Use the equation \eqref{eq:red} to write
\[
\pat (Z(x) -Z(y)) = \left[\Lambda Z (x) - \Lambda Z (y)   \right] +f \left[ \pax Z (x) Z (x) -  \pax Z(y) Z (y)\right] =:  I + f II.
\]
Due to translation-invariance of our problem, we can choose the reference point $0$ on $\So$ so that $x = \frac{\xi}{2}$, $y =- \frac{\xi}{2}$.
We write $\xi := x-y$. Via the differential formula in \eqref{1} and basic properties of the Poisson kernel, we arrive at
\begin{align*}
I \le \frac{1}{\pi} \int_0^\frac{\xi}{2} \frac{\omega(\xi + 2 \eta) +\omega(\xi - 2 \eta) -2 \omega (\xi) }{\eta^2} d \eta +  \frac{1}{\pi} \int_\frac{\xi}{2}^\infty \frac{\omega( 2 \eta + \xi) -\omega(2 \eta - \xi ) -2 \omega (\xi) }{\eta^2} d \eta.
\end{align*}
For more details of this computation, see \cite{KNS}, p. 222. Both terms above are nonpositive in view of concavity and nonnegativity of $\omega$. For $II$ we have
\begin{align*}
II = \pax Z (x) Z (x) -  \pax Z(y) Z (y)=  \frac{d}{dh}  \left( Z (x + h Z (x) ) - Z (y + h Z (y) ) \right)_{|h=0}  =:  \frac{d}{dh} g(h)_{|h=0}.
\end{align*}
Now, we use monotonicity of $\omega$ and \eqref{eq:twopt'} to compute $\frac{d}{dh} g(h)_{|h=0}$. Indeed,
\[
g(h) \le \omega (\xi + h | Z(x) - Z (y)| ) \le \omega (\xi + h \omega (\xi) ) \quad \text{ and } \quad g(0)  =  \omega (\xi).
\]
Consequently
\[
\frac{g(h) - g(0)}{h} \le \frac{ \omega (\xi + h \omega (\xi) )- \omega (\xi)}{h}  \implies II = \frac{d}{dh} g(h)_{|h=0} \le  \omega (\xi)  \omega' (\xi).
\]
Putting together estimates for $I$ and $II$ we obtain
\begin{align}
\pat (Z(x) -Z(y)) &\le 
 \frac{1}{\pi} \int_0^\frac{\xi}{2} \frac{\omega(\xi + 2 \eta) +\omega(\xi - 2 \eta) -2 \omega (\xi) }{\eta^2} d \eta\nonumber\\
 &+  \frac{1}{\pi} \int_\frac{\xi}{2}^\infty \frac{\omega( 2 \eta + \xi) -\omega( 2 \eta - \xi) -2 \omega (\xi) }{\eta^2} d \eta  + \Gamma \omega (\xi)  \omega' (\xi)=: I_1 + I_2 + I_3,\label{cd'}
\end{align}
with \begin{equation}\label{ass:wkb1}
\Gamma = \sup_{t \in [t_0, T]} f (t),
\end{equation}
recalling that  $f = f(t)$,
where $I_1, I_2$ are nonpositive. It remains to pinpoint  such $\omega \in \mathcal{O}$ that
 \begin{equation}\label{eq:iii22}
 I_1 + I_2 + I_3 < 0,
\end{equation}
since then \eqref{cd'} implies \eqref{cd}. 

Let us see what kind of dissipation (negative sign) may be provided by a typical modulus of continuity via $I_1$ and $I_2$. The formula behind  integral in $I_1$ resembles a second-order symmetric difference quotient. Therefore, using concavity of $\omega$, we write for $2 \eta \in [0, \xi]$
\[
\begin{aligned}
\omega (\xi + 2 \eta) - \omega(\xi) &\le  \omega'(\xi) 2 \eta , \\
\omega (\xi - 2 \eta) - \omega  (\xi) &= -\omega' (\xi) 2 \eta + \omega''(\theta) 2 \eta^2.
\end{aligned}
\]
Adding these formulas yields for $\theta \in [\xi - 2 \eta, \xi]$
\[
I_1 \le \frac{1}{\pi} \xi \omega'' (\theta).
\]
A typical modulus of continuity has increasing (and negative) second derivatives; then
 \begin{equation}\label{eq:iii22:p2}
I_1 \le \frac{1}{\pi} \xi \omega'' (\xi).
 \end{equation}
Let us focus on $I_2$.
Due to the concavity and nonnegativity of $\omega$, we have
\[\omega (2 \eta+ \xi ) -\omega (2 \eta - \xi) \le  \omega (2 \xi).
\]
If a nonnegative, strictly concave function $\omega$ satisfies
 \begin{equation}\label{eq:gam22}
 \omega (2 \xi) \le (1+ \gamma) \omega  (\xi) 
 \end{equation}
with $\gamma <1$ at some $\xi>0$, then
\[
I_2  \le - \frac{2 (1-\gamma)}{\pi} \frac{\omega (\xi) }{\xi}.
\]
Summing up,  for $\omega$ satisfying, in addition to previous assumptions,
\begin{itemize}
\item inequality \eqref{eq:gam22} (a typical behaviour of canonical nonnegative strictly concave functions off the origin),
\item that  its second derivatives are increasing,
\end{itemize} it holds
\[
 I_1 + I_2 + I_3 \le \frac{1}{\pi} \xi \omega'' (\xi) - \frac{2 (1-\gamma)}{\pi} \frac{\omega (\xi) }{\xi} + \Gamma \omega (\xi)  \omega' (\xi).
\]
Hence in order to have  \eqref{eq:iii22} we need for a large $\Gamma$
 \begin{equation}\label{eq:iii22:2}
 \xi \omega'' (\xi) - 2 (1-\gamma) \frac{\omega (\xi) }{\xi} + \pi \Gamma \omega (\xi)  \omega' (\xi) <0,
 \end{equation}
where we can throw out any of the two former summands, in view of nonnegativity of $I_1, I_2$. Throwing out $\xi \omega'' (\xi) $ lifts the need for the assumption that second derivatives of $\omega$ are increasing, whereas throwing out $- \frac{\omega (\xi) }{\xi}$ lifts the need for  \eqref{eq:gam22}.

 The simplest moduli of continuity are 
\[
\omega_0 (\xi) = \ln (a+ \xi)  - \ln a, \; a>0, \qquad \omega_1 (\xi) = \xi^\alpha, \ \alpha \in (0,1).
\]
Neither belongs to $\mathcal{O}$, since at $0$ $\omega''_0 > -\infty$, $\omega'_1 = \infty$. Let us consider then a slightly more complex
 \begin{equation}\label{eq:iii22:al}
\omega_2 (\xi) = \frac{\xi}{1+ K \xi^\alpha}, \; \alpha \in (0,1), \; K >0.
 \end{equation}
It fulfills all the requirements \emph{i.e.} $\omega_2 \in \mathcal{O}$, $\omega'''_2 \ge 0$, except for \eqref{eq:gam22} in vicinity of $0$. It means that we need here to throw out $- \frac{\omega (\xi) }{\xi}$ of \eqref{eq:iii22:2} and try to obtain 
 \begin{equation}\label{eq:iii22:3}
 \xi \omega_2'' (\xi) +  \pi \Gamma \omega_2 (\xi)  \omega_2' (\xi) <0
 \end{equation}
or equivalently
 \begin{equation}\label{eq:iii22:3b}
 (1 + K  (1-\alpha) \xi^\alpha) ( \pi \Gamma \xi^{1 - \alpha} - K \alpha) < K \alpha^2.
 \end{equation}
It holds for $\xi \le \xi_0 (\Gamma, K, \alpha)$. It means that choosing $\omega_2$ allows us to have  \eqref{eq:iii22:2} for $\xi \le \xi_0 (C, K, \alpha)$. 
 \begin{rem}\label{rem:emod}
One can conjecture the form \eqref{eq:iii22:al} of modulus of continuity, in relation to  \eqref{eq:iii22:2}, via integrating $ \xi \omega'' (\xi) + \pi \Gamma \omega (\xi)  \omega' (\xi) = 0$. See also Section \ref{sec:conc} for more details in context of relation to radially symmetric classical $2$d Smoluchowski-Poisson system.
 \end{rem}
 Let us focus now on the case $\xi \ge \xi_0 (C, K, \alpha)$. We will see shortly that there is in fact no difficulty to have  \eqref{eq:iii22:2}  for $\xi \ge \xi_0 (C, K, \alpha)$, provided  one allows to choose a modulus of continuity behaving differently in this regime. Then the eventual modulus will consist of two parts, where one needs to care about its smoothness at the point of the regime change. Having this in mind, observe that belonging to the family $ \mathcal{O}$ relies on the case $\xi \le \xi_0 (C, K, \alpha)$, hence for  $\xi \ge \xi_0 (C, K, \alpha)$ also the simplest moduli are admissible. Since $ K_1 \ln (K_2 \xi)$ fulfills  \eqref{eq:gam22} for $\gamma \ge \frac{\ln 2}{\ln (K_2 \xi_0)}$ and  $K_2 \xi_0 \ge 1$, one may make use of $- \frac{\omega (\xi) }{\xi}$ in  \eqref{eq:iii22:2}.  For small $K_1$'s, modulus $ K_1 \ln (K_2 \xi)$  suffices for
  \begin{equation}\label{eq:iii22:2D}
- 2 (1-\gamma) \frac{1 }{\xi} + \pi \Gamma \omega' (\xi) <0,
 \end{equation}
 hence for   \eqref{eq:iii22:2}. 
 
 All in all we see that a modulus of type
 \begin{equation}\label{eq:iii22:4}
\omega := \begin{cases} \frac{ \xi}{1 + K{ \xi}^\alpha} \quad &\text{ for } \xi \in [0, \xi_0), \\
 \ln (K_1 \xi) \quad &\text{ for } \xi \ge  \xi_0,  \\
\end{cases}
\end{equation}
 where $\alpha \in (0,1)$, provided it is smooth enough, is a reasonable choice to have  \eqref{eq:iii22:2}, thus \eqref{eq:iii22}. Let us fix $\alpha = \frac{1}{2}$. In the next sections we provide a precise choice of family of moduli, according to the just-presented motivation. 
 
 Let us only remark that knowing that one needs $\omega_2$ only for $ \xi \in [0, \xi_0)$, it is computationally more convenient to choose an equivalent modulus of type $\xi - K \xi^{1+\alpha}$, as in \cite{KNV}. We will stick however to the choice of type $\omega_2$, since it  additionally simplifies a nice insight into the comparison between our smoothness result and blowups of radially symmetric, classical Keller-Segel system in two dimensions -- see Section \ref{sec:conc}.
 
 \subsubsection{A concrete family of moduli}

For parameters $K, B, \xi_0$ let us define the family
\begin{equation}\label{eq:fam}
\omega_{K, B, \xi_0} (\xi) := \begin{cases} \frac{B \xi}{1 + K \sqrt{B \xi}} \quad &\text{ for } \xi \in [0, \xi_0), \\
C_{K,B, \xi_0} \ln (B\xi) \quad &\text{ for } \xi \ge  \xi_0,  \\
\end{cases}
\end{equation}
where 
\[
C_{K,B, \xi_0} = \frac{B \xi_0}{\ln (B \xi_0) ( 1+ K \sqrt{B \xi_0})},
\]
that implies continuity of $\omega_{K, B, \xi_0}$. This concrete family is motivated by  \eqref{eq:iii22:4}  and scaling $u_\lambda (x,t) = u (\lambda x, \lambda t)$ of the real-line case of \eqref{eq:red} with $f \equiv 1$. One has
\[
 \omega_{K, B, \xi_0}' (\xi)  = \begin{cases} B \frac{2 + K \sqrt{B \xi}}{2(1 + K \sqrt{B \xi})^2} \; &\text{ for } \xi \in [0, \xi_0), \\
C_{K,B, \xi_0} \xi^{-1} \; &\text{ for } \xi >  \xi_0,  \\
\end{cases} \qquad  \omega_{K, B, \xi_0}'' (\xi)  = \begin{cases} - K B^2 \frac{3 (B \xi)^{-\frac{1}{2}} + K}{4(1 + K \sqrt{B \xi})^3} \; &\text{ for } \xi \in [0, \xi_0), \\
- C_{K,B, \xi_0}  \xi^{-2} \; &\text{ for } \xi >  \xi_0.  \\
\end{cases}
\]
In particular, $\paxi \omega_{K, B, \xi_0} (0) = B$,  $\lim_{\xi \to 0} \paxi^2 \omega_{K, B, \xi_0} (\xi) = -\infty$.  To conclude that $\omega_{K, B, \xi_0} (\xi)  \in  \mathcal{O}$,  we need to show concavity of $\omega_{B}$. We have
\[
\omega_{K, B, \xi_0}' (\xi_0^-) \ge \omega_{K, B, \xi_0}' (\xi_0^+) \iff 2(1 +  K \sqrt{B \xi_0}) \le \ln(B \xi_0) (2 +  K \sqrt{B \xi_0})
\]
and for the latter it is sufficient to assume 
\begin{equation}\label{eqv20:1}
B \xi_0 \ge e^2.
\end{equation}
\subsubsection{Ruling out the two-point blowup scenario} 
Our aim is to choose parameters $K, B, \xi_0$ in \eqref{eq:fam} so that \eqref{eq:iii22} holds. 
As in \eqref{cd'} we obtain
\begin{equation}\label{cpl_ini3a}
\Gamma II \le I_3 = \Gamma \omega_{K, B, \xi_0} (\xi)  \omega_{K, B, \xi_0}' (\xi^-) \le \begin{cases} \Gamma B^2  \frac{ \xi (2 + K\sqrt{B \xi})}{2(1 + K \sqrt{B \xi})^3} \quad &\text{ for } \xi \in [0, \xi_0], \\
\Gamma C^2_{K, B, \xi_0} \xi^{-1} \ln (B \xi)  \quad &\text{ for } \xi >   \xi_0. \\
\end{cases}
\end{equation}
Observe adding case $\xi=\xi_0$ to the regime of  $\xi < \xi_0$ due to use of $ \omega (\xi^-)$. It may differ from $\omega (\xi^+)$ only in the spline point $\xi_0$, where  $ \omega (\xi^-) \ge  \omega (\xi^+)$ as long as \eqref{eqv20:1} holds.

Let us now consider the cases  $\xi \in [0, \xi_0)$ and  $\xi \ge  \xi_0$ separately.
\paragraph{(i) Case $\xi \in [0, \xi_0)$} We aim here at obtaining \eqref{eq:iii22:3}. Formula \eqref{eq:iii22:p2} yields
\[
I_1 \le \frac{1}{\pi} \xi \omega_{K, B, \xi_0}'' (\xi) = - \frac{1}{4 \pi}   K B^2   \xi \frac{3 (B \xi)^{-\frac{1}{2}} + K}{(1 + K \sqrt{B \xi})^3}.   
\]
This and \eqref{cpl_ini3a} implies that  \eqref{eq:iii22:3}  is equivalent to 
\[
K (3 (B \xi)^{-\frac{1}{2}} + K) > 2 \pi \Gamma (2 + K\sqrt{B \xi}) = 2 \pi  \Gamma  \sqrt{B \xi} (2 (B \xi)^{-\frac{1}{2}}  + K),
\]
for which in turn the condition
\begin{equation}\label{cpl_ini3}
K >  2 \pi  \Gamma  \sqrt{B \xi_0}
\end{equation}
is sufficient.

\paragraph{(ii) Case $ \xi \ge \xi_0$} We aim here at obtaining \eqref{eq:iii22:2D}. In fact, due to a possible jump at $\xi_0$, we need 
  \begin{equation}\label{eq:iii22:2Dn}
- 2 (1-\gamma) \frac{1 }{\xi} + \pi \Gamma \omega' (\xi^-) <0.
 \end{equation}
Using \eqref{eq:fam} and \eqref{eqv20:1}, we see that $\gamma = \frac{1}{2}$ is admissible in \eqref{eq:iii22:2Dn}.
Consequently, \eqref{eq:iii22:2Dn} is equivalent to $\pi \Gamma  < C_{K, B, \xi_0}^{-1}$ for $\xi > \xi_0$. In view of formula for $ C_{K, B, \xi_0}$ it is in turn equivalent to
\begin{equation}\label{cpl_ini4}
\Gamma \pi B \xi_0 < \ln (B \xi_0) (1+ K \sqrt{ B\xi_0}).
\end{equation}
For the case $ \xi = \xi_0$, \eqref{cpl_ini3} is again sufficient.

Summing up, we have gathered sufficient conditions so that \eqref{eq:iii22} holds for the family \eqref{eq:fam}. These are \eqref{eqv20:1},  \eqref{cpl_ini3}, \eqref{cpl_ini4}. Recall that having them fulfilled implies \eqref{cd}, thus contradicts \eqref{ass:au} as desired, provided the initial-data inequality \eqref{cpl_ev0} is valid. 

\subsubsection{Compliance with initial data} 
We are left with ensuring that \eqref{cpl_ev0} is kept via an appropriate choice of $K, B, \xi_0$. We have 
\[
|Z (x, t_0) - Z(y, t_0)| \le |\pax Z (t_0)|_\infty |x-y|.
\]
Take $\xi:= |x-y|$. Recalling our choice \eqref{eq:fam}, we see that in the case $\xi \in [0, \xi_0) $, for  \eqref{cpl_ev0} suffices
\begin{equation}\label{cpl_ini1}
\frac{B}{1 + K \sqrt{B \xi_0}} \ge |\pax Z (t_0)|_\infty.
\end{equation}
We split  the case $\xi \ge  \xi_0$ into $\xi \in [\xi_0, N \xi_0] $ and $\xi > N \xi_0$, where $N\geq1$ will be chosen later. Here, taking into account the form of $C_{K,B, \xi_0}$, for \eqref{cpl_ev0} 
suffices
\begin{equation}\label{cpl_ini2}
\begin{aligned}
\frac{B}{ ( 1+ K \sqrt{B \xi_0})}  > N  |\pax Z (t_0)|_\infty \quad &\text{ in the case } \xi \in [\xi_0, N \xi_0] \cr
\frac{B \xi_0}{ ( 1+ K \sqrt{B \xi_0})} \frac{\ln (B \xi)}{\ln (B \xi_0)}  >  2 |Z (t_0)|_\infty \quad &\text{ in the case } \xi > N \xi_0.
\end{aligned}
\end{equation}
It may happen that the half-period $L$ of the torus is smaller than $N \xi_0$ or even than $\xi_0$. Then, the corresponding conditions are not necessary.

\subsubsection{Conclusion of proof of  Lemma \ref{mainlemma}}  Summing up the previous two subsections, we need to choose such parameters $ K, B, \xi_0$ and $N$ (related to the splitting $N \xi_0$ in \eqref{cpl_ini2})  so that requirements  \eqref{eqv20:1},  \eqref{cpl_ini3}, \eqref{cpl_ini4},  \eqref{cpl_ini1}, \eqref{cpl_ini2} are satisfied. Let us first make the following choices for $\xi_0$ and $K$:
\[ B \xi_0 = e^2, \quad K =  4 \pi  \Gamma  \sqrt{B \xi_0} =  4 \pi  \Gamma e   , \]
so that  \eqref{eqv20:1} and \eqref{cpl_ini3} are fulfilled. We are left with free parameters $B$ and $N$ and we need to comply with \eqref{cpl_ini1}, \eqref{cpl_ini2} and \eqref{cpl_ini4}. Respectively, they take now the form
\[
\frac{B}{1 + 4 \pi  \Gamma e^2} \ge |\pax Z (t_0)|_\infty, 
\]
\[
\begin{aligned}
\frac{B}{ 1+ 4 \pi  \Gamma e^2} > N  |\pax Z (t_0)|_\infty  & \quad \text{ in the case } \xi \in [\xi_0, N \xi_0] \\
\frac{e^2}{ ( 1+  4 \pi  \Gamma e^2)} \ln (B \xi) > 2 |Z ( t_0)|_\infty & \quad \text{ in the case } \xi > N \xi_0
\end{aligned}
\]
and 
\[ \Gamma \pi e^2 <  2 (1+ 4 \pi  \Gamma  e^2). \]
The last one holds automatically. For the condition containing $|Z (t_0)|_\infty$, it suffices to choose $N \ge 1$ such that
\[\frac{1}{ ( 1+  4 \pi  \Gamma e^2)} \ln (N e^2) = |Z (t_0)|_\infty. \]
This choice fixes $N$. Finally we choose 
\begin{equation}\label{Bchoi}
B = 2N  |\pax Z (t_0)|_\infty  { ( 1+ 4 \pi  \Gamma e^2 )} 
\end{equation}
which suffices for the two  requirements involving $|\pax Z (t_0)|_\infty$. 

We have met all the conditions   \eqref{eqv20:1},  \eqref{cpl_ini3}, \eqref{cpl_ini4},  \eqref{cpl_ini1}, \eqref{cpl_ini2}. It means that for an arbitrary $Z(t_0) \in C^\infty (\So)$ we have found a modulus $\omega \in \mathcal{O}$ such that it is kept by $Z(t_0) \in C^\infty (\So)$ and is not violated for $t \ge t_0$ over the evolution. Hence Lemma  \ref{reduction} gives us the thesis of our main Lemma \ref{mainlemma}.\qed
\subsection{Conclusion of proof of Theorem \ref{th2}}
We have now all ingredients needed for our proof of Theorem \ref{th2}. Lemmas \ref{th:loc}, \ref{bur:analytic} imply that any initial datum in $H^s$, $s\ge1$ gives rise to a locally-in-time smooth solution. Lemma \ref{mainlemma} says that it can be continued for any $T< \infty$. Theorem \ref{th2} is proved.
\qed

\section{Proof of Theorem \ref{asymptotics}}\label{sec:th3}
 An $L^p$-estimate for $Z$ solving the Burgers equation \eqref{eq:red} implies an exponential decay. On the other hand, $ \pax Z \; (= (u-m)  e^{-m \chi t}) $ may a priori grow double-exponentially:  compare inequality \eqref{omega'}, where we use  $\omega'_{K, B, \xi_0} (0) = B$, with $B \sim  N \Gamma \sim e^{\Gamma} \Gamma $ via  \eqref{Bchoi}  and  $\Gamma  \sim e^{m \chi t}$ by definition.  Nevertheless, in the following we are able to show that $ (u-m)$ decays exponentially for sufficiently small $\chi m$.
 
Recall that we restrict ourselves here to the periodic torus $[-\pi, \pi]$. As before, we denote the full Hilbert norm as $| f |_{H^k} = | f |_k$ with $k < \infty$, where $k=0$ stands for the $L^2$ norm, and $| f |_\infty$ is reserved for the supremum norm.  Since we will work in what follows with homogeneous Hilbert norms, for clarity, we will write them as $| \Lambda^k f| _0$.

 For zero mean functions, we will need the Poincar\'e inequality
\begin{equation}\label{eqdecayP}
|f|_0^2\leq |\Lambda^\frac{1}{2} f|_0^2
\end{equation}
and  the interpolation inequalities
\begin{equation}\label{eqdecayI}
|f|_{L^4}^4\leq C_I |f|_0^2|\Lambda^\frac{1}{2} f|_0^2, 
\end{equation}
\begin{equation}\label{eqdecayI2}
 | f|^2_\infty \leq 2  | f|_0  |\Lambda f|_0.
\end{equation}
Recall that by assumption we have a fixed number $\chi m <1$. $C$ is a generic constant that may vary between lines.

\subsection{Zero-order decays}

Testing equation \eqref{1_ks'} with $W$, we obtain 
$$
\frac{1}{2}\frac{d}{dt}|W (t) |_0^2\leq \chi m |W(t)|_0^2-|\Lambda^\frac{1}{2} W(t)|_0^2.
$$
Hence
\begin{equation}\label{eqdecay}
|W(t)|_0 \leq e^{\left(-1+\chi m\right)t}   |W(0)|_0,
\end{equation}
where we have used the Poincar\'e inequality \eqref{eqdecayP}.
Using  the assumption $\chi m<1$, we can conclude from  \eqref{eqdecay} that
\begin{equation}\label{eqdecay20}
\int_0^\infty |\Lambda^\frac{1}{2} W(s)|_0^2ds\leq C.
\end{equation}

Let us denote by  $f(x^*)=\max_{x \in \So} f (x)$. Using the formula \eqref{1sin} for $\Lambda$ we obtain
$$
\Lambda f(x^*)=\frac{1}{2\pi}\text{P.V.}\int_{-\pi}^\pi \frac{f(x^*)-f(x^*-y)}{\sin^2\left(y/2\right)}dy\geq f(x^*),
$$
for  a zero mean function $f$, since $\sin^2 \le 1$.
This inequality and tracking the spatial maximum of $W(t)$, \emph{i.e.} $W(x^*_t, t)$ via  \eqref{1_ks'}, we obtain\footnote{For more details of this procedure, including its rigorousness, compare \cite{AGM, BG, GO}.} the $L^\infty$-decay
$$
|W(t)|_{L^\infty}\leq e^{\left(-1+\chi m\right)t} |W(0)|_{L^\infty}.
$$ 
The $L^2$ and  $L^\infty$ decays allow us to choose $1\ll T^*$ such that
\begin{equation}\label{eqdecayT}
\sup_{t\geq T^*}\chi\left(|W(t)|_{0}+|W(t)|_{L^\infty}\right)\leq   \frac{1 - \chi m}{2}
\end{equation}
 for any $t\geq T^*$.
\subsection{Half-order decays}
Testing \eqref{1_ks'} with $\Lambda W$, one obtains
$$
\frac{1}{2}\frac{d}{dt}|\Lambda^\frac{1}{2} W(t)|_{0}^2\leq |\Lambda W(t)|_{0}^2(-1+\chi |W(t)|_{L^\infty})+\chi m|\Lambda^\frac{1}{2} W(t)|_{0}^2,
$$
where we used the fact that $\Lambda = \pax H$ and the $L^2$-isometry of the Hilbert transform $H$ for periodic, zero mean functions. 
This estimate, together with \eqref{eqdecay20} and the choice \eqref{eqdecayT} 
gives
$$
\sup_{t\geq T^*}|\Lambda^\frac{1}{2} W(t)|_{0}^2 + \int_{T^*}^\infty |\Lambda W(s)|_0^2ds\leq C
$$
hence
\begin{equation}\label{eqdecay2}
 \int_0^\infty |u(s)-m|_{0}^2ds = \int_0^\infty |\Lambda W(s)|_{0}^2ds \leq C,
\end{equation}
where the equality follows from $\pax W = u-m$, see \eqref{smallW} and the inequality for $t \le T^*$ -- from Theorem \ref{th2}.

The $L^2$-estimate for the Keller-Segel problem  \eqref{1_ks_u} - \eqref{1_ks_v} implies
\begin{align*}
\frac{1}{2}\frac{d}{dt}|u(t) |_{0}^2& \le  -|\Lambda^\frac{1}{2} u(t)|_{0}^2+\frac{ \chi}{2} |u(t)|_{L^3}^3
\leq -|\Lambda^\frac{1}{2} u(t)|_{0}^2+\frac{\chi}{2}|u(t)|_{L^2}|u(t)|_{L^4}^2\\
&\leq -|\Lambda^\frac{1}{2} u(t)|_{0}^2+\frac{C_I \chi}{2}|u(t)|^2_{0}|\Lambda^\frac{1}{2} u(t)|_{0},
\end{align*}
where for the last inequality we used the interpolation \eqref{eqdecayI}. Hence Young's and Gronwall's inequalities, together with the bound \eqref{eqdecay2}, yield
$$
\sup_{t \ge T^*}|u(t)|_{0} \leq C
$$ 
and consequently, using for  $t \le T^*$ Theorem \ref{th},
\begin{equation}\label{eqdecay3}
\sup_{t \in \er }|u(t)|_{0} +  \int_0^\infty |\Lambda^\frac{1}{2}u(s)|_{0}^2ds\leq C.
\end{equation}
\subsection{First-order decays}
Testing equation \eqref{1_ks_u} with $\Lambda u$, we obtain
$$
\frac{1}{2}\frac{d}{dt}|\Lambda^\frac{1}{2} u(t)|_{0}^2\leq -|\Lambda u(t)|_{0}^2-\chi\int_{-\pi}^\pi (\partial_x u\partial_x v\Lambda u)(t)+\chi\int_{-\pi}^\pi  (u(u-m)\Lambda u)(t).
$$
For the middle term on the right-hand side, we use the inequality
$$
-\chi\int_{-\pi}^\pi  (\partial_x u\partial_x v\Lambda u)(t)\leq \chi |\Lambda u(t)|_{0}^2|\partial_x v(t)|_{L^\infty}
$$
in  tandem with
\begin{equation}\label{eqdecay31}
|\partial_x v|_{L^\infty}=|W|_{L^\infty}\leq \frac{1}{2 \chi},
\end{equation}
compare \eqref{eqdecayT}.

For the last term we write
$$
\chi\int_{-\pi}^\pi  (u(u-m)\Lambda u)(t)=-\chi m|\Lambda^\frac{1}{2} u(t) |_{0}^2+ \chi |\Lambda u (t) |_{0}|u (t)|^2_{L^4} \le \frac{1}{4} |\Lambda u (t) |_{0}^2 + C_I \chi (\chi |u(t)|^2_{0} + m) |\Lambda^\frac{1}{2} u(t)|^2_{0},
$$
where for the first inequality we used the interpolation \eqref{eqdecayI} and for the second one the Poincar\'e inequality \eqref{eqdecayP}.

Together, we obtain
$$
\frac{1}{2}\frac{d}{dt}|\Lambda^\frac{1}{2}u (t) |_{0}^2\leq -\frac{1}{4} |\Lambda u (t) |_{0}^2 + C |\Lambda^\frac{1}{2}u (t) |_{0}^2 (1+  |\Lambda^\frac{1}{2}u (t) |_{0}^2).
$$
This, with aid of \eqref{eqdecay3}, implies, similarly as before
\begin{equation}\label{eqdecay21}
\sup_{t \in  \er}|\Lambda^\frac{1}{2}u (t) |_{0} + \int_0^\infty   |\Lambda u (s) |_{0}^2 ds \leq C.
\end{equation}
Using $\pax W = u-m$ and Theorem \ref{th2} for $t \le T^*$, we also get
\begin{equation}\label{eqdecayZ}
\sup_{t \in \er}| W (t) |_\frac{3}{2} \leq C.
\end{equation}
\subsection{Second-order decays}
Testing equation \eqref{1_ks_u} with $\Lambda^2 u$, we obtain
$$
\frac{1}{2}\frac{d}{dt}|\Lambda u(t)|_{0}^2\leq -|\Lambda^\frac{3}{2} u(t)|_{0}^2-\chi\int_{-\pi}^\pi (\partial_x u\partial_x v\Lambda^2 u)(t)+\chi\int_{-\pi}^\pi  (u(u-m)\Lambda^2 u)(t).
$$
For the middle term on the right-hand side, using again \eqref{eqdecay31}, we write
\begin{multline}
-\chi\int_{-\pi}^\pi  (\Lambda^\frac{1}{2} ( \partial_x u\partial_x v )\Lambda^\frac{3}{2} u)(t)\leq \chi |\Lambda^\frac{3}{2} u(t)|_{0}^2 \left(|\partial_x v(t)|_{L^\infty} + \frac{1}{4 \chi} \right) + \chi^2 |(\partial_x u\partial_x \Lambda^\frac{1}{2}  v)(t)|^2_0 \\
 \le \frac{3}{4}  |\Lambda^\frac{3}{2} u(t)|_{0}^2 + C \chi^2 |\partial_x u (t) |^2_0, 
\end{multline}
where for the last inequality we used $|\pax \Lambda^\frac{1}{2}  v (t)|^2_\infty \le C$ via \eqref{eqdecayZ}.

For the last term on the right-hand side it holds
$$
\chi\int_{-\pi}^\pi  (u(u-m)\Lambda^2 u)(t)  =- \chi m|\Lambda u(t) |_{0}^2 + \chi \int_{-\pi}^\pi ( \Lambda^\frac{1}{2}  (u^2) \Lambda^\frac{3}{2} u)(t)    \le \frac{1}{8}  |\Lambda^\frac{3}{2} u(t)|_{0}^2 +
C \chi^2 |\Lambda^\frac{1}{2}u(t)|^2_{0} |u(t)|^2_{L^\infty}.$$
Altogether, using also interpolation \eqref{eqdecayI2}, we obtain
\[
\frac{d}{dt}|\Lambda u(t)|_{0}^2 + \frac{1}{4} |\Lambda^\frac{3}{2} u(t)|_{0}^2 \le C \chi^2 (|\partial_x u (t) |^2_0 +  |\Lambda^\frac{1}{2}u  (t)|^{2}_{0}   | u  (t)|_0  |\Lambda u  (t)|_0  ).
\]
Therefore, using \eqref{eqdecay21}, we arrive at
\begin{equation}\label{eqdecay41}
\sup_{t \in  \er}|\Lambda u (t) |_{0} + \int_0^\infty   |\Lambda^2 u (s) |_{0}^2 ds \leq C.
\end{equation}
Hence
\begin{equation}\label{eqdecayZ2}
\sup_{t \in \er}| W (t) |_2 \leq C.
\end{equation}

\subsection{Asymptotics}

Using interpolation in Sobolev spaces and \eqref{eqdecay} with \eqref{eqdecayZ2}  one has for any $t>0$, $\alpha <1$ 
$$
|(u-m)(t)|_{\alpha} = | W (t)|_{1+\alpha} \leq C |W (t) |^{\frac{1-\alpha}{2}}_{0}|W(t) |_2^{\frac{1+\alpha}{2}} \le \Sigma e^{ (1-\alpha) \left(-1+\chi m\right) t}.
$$
for a finite $\Sigma$.
\qed


 \section{Concluding remarks}\label{sec:conc}
We have showed that the critical fractal Keller-Segel problem in the periodic setting admits global-in-time, smooth solutions. This closes an open question, posed in \cite{BouCal} and present earlier in \cite{Escudero}, at least in the periodic setting. Apart from this, we believe that both
\begin{itemize}
\item the observation that solutions to parabolic-elliptic Keller-Segel system are given as derivatives of the solutions to a Burgers-type equation, and
\item the generalization of \cite{KNS} methodology over equations with no scaling 
\end{itemize}
may be useful for further studies. 
Additionally, we obtained the steady-state asymptotics of solutions to \eqref{1_ks_u} -  \eqref{1_ks_v}, as long as quantity $\chi m$ is small enough. 

In the remainder of this final section we compare the fractional Keller-Segel system in our periodic setting  with the real-line case. Next, we mention some interesting relations between the radially symmetric, classical two-dimensional Keller-Segel system (that gives rise to blowups for large masses) and our system. Finally we list three open questions related to   \eqref{1_ks_u} -  \eqref{1_ks_v}.

\subsection{Comparison with equations on $\er$.} 
\subsubsection{Role of $m$}
In the real-line case it is common to consider
\begin{equation}\label{1_ks_v_r}
0= \pax^2 v + u \quad  \text{ in } \;(0, T) \times  \er
\end{equation}
as the equation for $v$, compare Escudero \cite{Escudero} and Bournaveas \& Calvez \cite{BouCal}. In fact, the proper periodic counterpart of \eqref{1_ks_v_r} is our  \eqref{1_ks_v}  and not \eqref{1_ks_v_r} on $\So$. Let us explain this matter. 

Firstly, when one considers a family of problems \eqref{1_ks_u} -   \eqref{1_ks_v} on periodic tori $[-L, L]$, all with the same total initial mass $\int_{-L}^L u_0 (x) dx$ and takes $L \to \infty$, then \eqref{1_ks_v}  goes (formally) to \eqref{1_ks_v_r}. Hence \eqref{1_ks_v} is a legitimate periodic counterpart of \eqref{1_ks_v_r}. 

Moreover, if we had dropped $m$ in \eqref{1_ks_v} (in the periodic case), integration by parts in the resulting Poisson equation would force $\dashint_{\Ss^1} u(x, t)dx = 0$. Since the applications call for nonnegative $u_0$ and $u$, the dynamics would become trivial. This is not the case on the real line, so we may consider the simplest possible \eqref{1_ks_v_r} there.

Finally, equation \eqref{1_ks_v} can be seen as the elliptic simplification of
\begin{equation}\label{1_ks_v_p}
\varepsilon \pat v= \pax^2 v + u -  m,
\end{equation}
appearing for instance in \cite{BurCie} by Burczak,  Cie\'slak \& Morales-Rodrigo. Observe that \eqref{1_ks_v_p}  formally yields $\dashint_\So v (x, t) dx \equiv 0$ by the mass conservation. Taking limit $\varepsilon \to 0$ motivates our choice of  \eqref{1_ks_v} as well as our zero-mean disambiguation \eqref{zeroF}.
\subsubsection{Possibility of an analogue of Theorem \ref{th} on $\er$}
Since the original blowup conjecture of \cite{BouCal} concerns the real-line case, it would be natural to provide an analogue of Theorem \ref{th} on $\er$, \emph{i.e.} for \eqref{1_ks_u} - \eqref{1_ks_v_r}. To this end one needs however to deal with the following matter, at least within our approach. Nonnegativity of initial datum $u_0$ implies that the direct analogue for respective fractional-Burgers-initial-datum is non-integrable, compare \eqref{bur:choi} with $m \equiv 0$. This is merely a technical matter that we deal with in our incoming paper.

 \subsection{ Relation to the radially symmetric, classical two-dimensional Keller-Segel system}
Let us consider the classical Keller-Segel system in two-dimensions on $(0, T) \times \Omega$
 \[
\begin{aligned}
\pat u  - \Delta u &= - \chi \nabla \cdot (u \nabla v), \\
- \Delta v &=  u.
\end{aligned}
\]
Under assumption of radial symmetry (it suffices to have radially symmetric domain and data to conserve radial symmetry over the evolution, due to rotation invariance), it is common to introduce a one-dimensional primitive-type equations. Let us focus on the case of $\Omega$ being the unit disc $B$, since we are dealing in this paper with bounded domains. The cumulative mass $Q (r,t) = \frac{1}{2 \pi} \int_{|x| \le r} u (x,t) d x$ follows
\[
\pat Q  - \partial_r^2 Q + r^{-1}  \partial_r Q  = \chi  r^{-1}  Q \partial_r Q.
\]
A more striking resemblance to our case can be seen for the new variable $S (r,t) = Q (\sqrt{r},t)$. It is governed by
\[
\pat S  - 4 r \partial_r^2 S  =  2\chi S  \partial_r  S .
\]
Firstly, the above equation produces blowups for certain initial data with supercritical mass $S_0(1) > \frac{4}{\chi}$, so, apparently, near the origin the dissipation $r \partial_r^2$ is weaker than $\Lambda u$ of \eqref{eq:red}. More importantly, the identity for stationary solutions\footnote{Non-trivial stationary solutions exist for any subcritical or critical mass here, in contrary to the full-space case, where they exist only for critical masses. For this and more details on the results presented in this subsection, compare papers by Biler, Karch, Lauren\c{c}ot \& Nadzieja:  \cite{BKLNdisc} for disc as a domain and \cite{BKLNplane} for the full-space case.}  $S_\infty$
\[
 r \partial_r^2 S_\infty   +  \frac{\chi}{2} S_\infty   \partial_r  S_\infty  = 0
\]
coincides with inequality \eqref{eq:iii22:3} up to constants and yields
\[
S_\infty (r) = \frac{4}{\chi} \frac{r}{ \frac{4 }{\chi S_0(1) }-1 + r}.
\]
For subcritical masses $S_0(1) < \frac{4}{\chi}$ it takes the form of the $\alpha =1$ endpoint of the modulus of continuity $\omega_2$, recall \eqref{eq:iii22:al} and Remark \ref{rem:emod}. However, observe that such $\alpha =1$ endpoint of $\omega_2$ does not belong to the class $\mathcal{O}$.
 \subsection{Open questions}
Concerning the system \eqref{1_ks_u} -  \eqref{1_ks_v}, we have three questions in mind. 

Firstly, it would be interesting to know whether the size restriction on $\chi m$ in Theorem \ref{asymptotics} is merely a technicality or there is nontrivial infinite-time dynamics for large initial masses. Preliminarily, we are in favor of the latter case.

Secondly, the results of \cite{DKSV} on smoothness in the finer scale of fractional dissipations can be automatically transferred to our case. An interesting open matter in this context may be to analyze a possibility of the threshold smoothness/blowup dissipation (or threshold size of the initial data) in this finer scale.
 
Finally, what is probably of main interest, is  to generalize our result over wider class of Keller-Segel systems, especially over the doubly parabolic case. This is a subject of our current research.

 \subsection*{Acknowledgement}
 JB thanks Tomasz Cie\'slak and Grzegorz Karch for fruitful discussions.


\bibliographystyle{abbrv}

\end{document}